\documentclass[12pt]{article}
\usepackage{pgfplots}
\usepackage{times}
\usepackage{amsmath,amsfonts, amstext,amssymb,amsbsy,amsopn,amsthm}
\usepackage{dsfont}
\usepackage{esint}
\usepackage{graphicx}   % for figures
\usepackage{hyperref}
\usepackage[all]{xy}
\usepackage{color}
\usepackage{tikz}
\usepackage{lineno,hyperref}

\newcommand{\fr}{\displaystyle\frac}
\newcommand{\jf}{\displaystyle\int}
\newcommand{\lt}{\left}
\newcommand{\rt}{\right}
\newcommand{\lm}{\lambda}

\newcommand{\mb}{\mbox}

\newcommand{\tm}{\times}

\newcommand{\pl}{\partial}

\newcommand{\abs}[1]{\lvert#1\rvert}
\newcommand{\Abs}[1]{\left\lvert#1\right\rvert}

\newcommand{\R}{\mathbb{R}}
\newcommand{\cL}{\mathcal{L}}

\newtheorem{theorem}{Theorem}[section]
\newtheorem{lemma}[theorem]{Lemma}

\newtheorem{corollary}[theorem]{Corollary}
\theoremstyle{definition}
\newtheorem{definition}[theorem]{Definition}
\theoremstyle{remark}
\newtheorem{remark}{Remark}[section]
\newtheorem{counterexample}{Counterexample}

\theoremstyle{remark}

\numberwithin{equation}{section}
\allowdisplaybreaks[4]

\begin{document}
\title{Liouville type theorems for  dual nonlocal evolution equations involving Marchaud derivatives}

%\author{Wenxiong Chen }
%\address{Department of Mathematical Sciences, Yeshiva University, New York, NY,  10033 USA}
%\email{wchen@yu.edu}
%\author{Yahong Guo}
%\address{School of Mathematical Sciences,
%Shanghai Jiaotong University,
%Shanghai, 200240, P.R. China}
%\email{guoyahong1995@outlook.com}

%\author{Lingwei Ma}
%\address{School of Mathematical Sciences, Nankai University, Tianjin, 300071, P.R. China}
%\email{mlw1103@outlook.com}

\author{Yahong Guo, Lingwei Ma, Zhenqiu Zhang\thanks{Corresponding author}}
%\address{School of Mathematical Sciences and LPMC, Nankai University,
%Tianjin, 300071, P.R. China}
%\email{zqzhang@nankai.edu.cn}
%\footnote{*Corresponding author.}

\date{\today}
\maketitle
\begin{abstract}
In this paper, we establish a Liouville type theorem for the homogeneous dual fractional parabolic   equation
\begin{equation*}
  \partial^\alpha_t u(x,t)+(-\Delta)^s u(x,t) = 0\ \ \mbox{in}\ \ \mathbb{R}^n\times\mathbb{R},
\end{equation*}
where $0<\alpha,s<1$.
By employing
a method of Fourier analysis, we prove that all solutions  in the sense of distributions must be affine, and hence  constant in the case $0<s\leq\frac{1}{2}.$  In the case $\frac{1}{2}<s < 1$, under the asymptotic assumption
$$\liminf_{|x|\rightarrow\infty}\frac{u(x,t)}{|x|^\gamma}\geq 0 \; ( \mbox{or} \; \leq 0)  \,\,\mbox{for some} \;0\leq\gamma\leq 1,  $$
 all the solutions  must be constant.
 Our result includes the previous Liouville theorems  on $s$-harmonic
  functions \cite{CDL}
 as special cases and it is still  novel even restricted to one-sided  Marchaud fractional equation.
Our methods  can be applied to a
 variety of dual nonlocal parabolic problems.

As an application of the above Liouville theorem, we establish an equivalence between the pseudo-differential equations involving Marchaud fractional derivatives and the corresponding integral equations.
\bigskip


\bigskip

\end{abstract}

\maketitle

\section{Introduction}\label{sec1}

In this paper, we investigate solutions of the  dual fractional parabolic equations
\begin{equation}\label{B}
\partial^\alpha_t u(x,t)+(-\Delta)^s u(x,t) = 0\ \ \mbox{in}\ \ \R^n\times\R .
\end{equation}
Under mild conditions, we prove that all solutions are constant and thus establish a Liouville type
 theorem for this problem.

Here, the one-sided nonlocal time derivative $\partial_t^\alpha$,   sometime also denoted by $D_{\rm{left}}^\alpha$, is defined as
\begin{equation}\label{1.00}
\partial^\alpha_t u(x,t)=C_\alpha\jf_{-\infty}^t\fr{u(x,t)-u(x,\tau)}{(t-\tau)^{1+\alpha}}d\tau,
\end{equation}
with $0<\alpha<1,  \mb{and~} C_\alpha:=\frac{1}{\Abs{\Gamma(-\alpha)}}$ with $\Gamma$ being the Gamma function.
This fractional time derivative
 is known as the Marchaud fractional (left) derivative  of order $\alpha$, depending on the values of the function  from the past time.
    %Such fractional time derivative arises in various physical phenomena, such as particle systems with sticking and trapping phenomena, magneto-thermoelastic heat conduction, plasma turbulence, etc. (cf. [25, 26, 27]).
  The spatial nonlocal elliptic pseudo-differential operator, the fractional Laplacian $(-\Delta)^s$
is defined by
 \begin{equation}\label{1.0}
(-\Delta)^su(x,t)= C_{n,s}P.V.\jf_{\R^{n}}\fr{u(x,t)-u(y,t)}{\abs{x-y}^{n+2s}}dy,
\end{equation}
where $0 < s < 1$, $C_{n,s} $ is a normalization positive constant and $P.V.$ stands for the Cauchy principal
value.
We notice that the operator $\partial_t^\alpha+(-\Delta)^s$ in equation \eqref{B} is nonlocal  in both space and time,
 since the value of $[\partial_t^\alpha+(-\Delta)^s] u$ at a given point $(x,t)$ depends on the values of $u$ over
  the whole space $\mathbb{R}^n$  as well as on all the past time before $t$.
% Moreover, it should be noted that as $s,\alpha$ tends to $1$  from the left side, the dual fractional parabolic operator $\partial_t^\alpha+(-\Delta)^s$ converges to the local heat operator $\partial_t-\Delta$ (cf. \cite{FNW}).
Moreover, this dual nonlocal operator  is reduced to the local heat
 operator in the sense that, for each $(x,t)\in \R^n\times\R$, \[[\partial_t^{\alpha} + (-\Delta)^s]u(x,t)\to[\partial_t - \Delta]u(x,t),\] as \(\alpha \to 1^-\) and \(s \to 1^-\).

The dual space-time nonlocal operator arises   in    various physical and probabilistic models,
 including
 magneto-thermoelastic heat conduction, plasma turbulence,  anomalous diffusion,
  chaotic dynamics as discussed in \cite{dCN1,dCN,EE,KBS,Z}.
 A typical application of equation \eqref{B}
is presented in \cite{dCN} to simulate transport of tracer particles in plasma, where the function $u$ is the
 probability density for tracer particles representing the likelihood of finding a particle
at time $t$ and position $x$. In this case, the fractional diffusion operator $(-\Delta)^s$  accounts for
avalanche-like transport that may occur, while the one-sided  fractional time derivative  $\partial^\alpha_t$ accounts for
trapping effect of turbulent eddies. %describes the entrapment of  of trace particles within turbulent eddies.
 % In applications within financial field, it can also be used to model the waiting time between transactions is correlated with the ensuring price jump (cf. \cite{RSM}).
The dual fractional kinetic equations \eqref{B} can also be derived from the context of continuous-time random walks \cite{MK},
which a generalization of  Brownian random walks that describe particles experience uncorrelated
random movements at fixed time intervals.
%  another model originates from   continuous-time random walks \cite{MK}. This dual fractional kinetic model here introduces non-locality in time, leading to historical dependence  due to unusually long waiting times, and non-locality in space,
%explaining exceptionally large jumps connecting distant regions,  such as L$\acute{e}$vy flights.
A fractional diffusion operator corresponds to a diverging jump length variance
in the random walk, and a fractional time derivative arises when the characteristic
waiting time diverges. A fractional time derivative is used to model processes that exhibit memory effects.
%Recently,  Guo, Ma and Zhang \cite{GMZ} proved that the bounded entire solutions of the homogeneous dual fractional parabolic  equation \eqref{B} must be constant. Nevertheless, it remains uncertain whether the Liouville theorem for dual fractional equation \eqref{B} still holds when weakening the two-sided boundedness condition. This is the primary objective of our research here.

Liouville type theorems play  crucial roles in deriving a priori estimates and in establishing
qualitative properties of solutions, including their existence, nonexistence, and uniqueness. Therefore,
 these type of theorems  has been widely studied for various types of local or nonlocal elliptic and parabolic equations using different
  methods, such as  Harnack inequalities, blowing up and compactness arguments, as well as
  Fourier analysis (cf. \cite{BKN, CDL, CW,Fa,GMZ,Hir,Mos,MGZ,  Ser,WuC} and the
references therein).

With respect to the classical fractional elliptic equation
 \begin{equation}\label{C}
 (-\Delta)^s u(x)=0~\mb{in}~\R^n,\end{equation}
Bogdan, Kulczycki and  Nowak \cite{BKN} showed that if u is a nonnegative solution of equation \eqref{C}, then it must be constant. Subsequently, Chen, D' Ambrosio and Li \cite{CDL} weaken
this one-sided boundedness condition to% $u\in\cL_{2s}(\R^n)$ and
\begin{equation}\label{asymp}
 \liminf_{|x|\rightarrow\infty}\frac{u(x)}{|x|^\gamma}\geq 0 \,\,\mbox{for}\,\, 0\leq\gamma\leq\min\{1,2s\}.
\end{equation}

As far as we know, there is still a lack of research on the Liouville-type theorem  for one-sided Marchaud fractional  equation\begin{equation}\label{D}
 \partial_t^\alpha u(t)=0~\mb{in}~\R.\end{equation} In recent work \cite{GMZ}, the authors employed a direct perturbation method to prove that the bounded solutions of equation \eqref{D} must be constant.
Then  it is natural to ask:
% { \em Can~ this~ boundedness ~condition ~ be~ relaxed? }
 \[{ Can~ this~ boundedness ~condition ~ be~ relaxed? }\]
 Here we answer this question affirmatively  as a special case of our main theorem.

 In the context of nonlocal  parabolic equation,
 %Widder \cite{Wid} first proved that all bounded solutions $u(x, t)$ must be constant for $\alpha=1,s=1$ , then this boundedness condition was relaxed to one-side bounded-ness with an  additional growth condition by Hirschman \cite{Hir}; Later on,
Serra \cite{Ser} showed that all bounded solutions  of  fractional parabolic equation\begin{equation*}
     \partial_tu(x,t)+(-\Delta)^s u(x,t)=0 \; ~\mb{in}~ \mathbb{R}^n\times(-\infty,0]
  \end{equation*} must be constant under  the following   growth conditions
$$\|u\|_{L^\infty(Q_R^{2s})}\leq C R^\beta\,\, \mbox{for all}  \,\, R\geq1 \,\,\mbox{and some}\,\, 0<\beta<\min\{1, 2s\},$$
where $Q_R^{2s}=\{(x,t)\mid |x|< R \,\,\mbox{and}\,\, -R^{2s}<t<0\}$ denotes a parabolic cylinder.

However, for dual fractional parabolic equations \eqref{B} simultaneously involving the Marchaud derivatives $\partial_t^\alpha$
and the fractional Laplacian $(-\Delta)^s$, there have seen very few results in this respect except a recent one in
\cite{GMZ} which asserts that  all bounded classical  entire solutions of  \eqref{B}  must be constant.
Then, is this result  still true for unbounded   tempered distributions? % ,such as $\mathcal{L}_{2s,\alpha}(\mathbb{R}^n\times\mathbb{R})$,.
To answer the question is  the primary objective of our research here. By employing  Fourier analysis on tempered distributions, we prove that all the solutions of \eqref{B} in the sense of distributions must be affine, and hence constant in the case $0<s\leq \frac{1}{2}.$ While in the case  $\frac{1}{2}<s<1$, under the mild assumption
\begin{equation}\label{asym}\liminf_{|x|\rightarrow\infty}\frac{u(x,t)}{|x|^\gamma}\geq 0 \; ( \mbox{or} \; \leq 0)  \,\,\mbox{for some} \;0\leq\gamma\leq 1,\end{equation}
 all the solutions  must also be constant. In particular in the case $u$ is independent of $t$, our equation \eqref{B} reduced to fractional Laplace equation \eqref{C}, and therefore our results here contain the results in both \cite{BKN} and \cite{CDL} as special cases.

In order to apply the Fourier Analysis to the distributional solutions,
 we first  choose  an  applicable   space of tempered distributions
$$\mathcal{L}_{2s,\alpha}(\mathbb{R}^n\times\mathbb{R})=\left\{u(x,t) \in L^1_{\rm loc} (\mathbb{R}^n\times\mathbb{R}) \mid \int_{-\infty}^{+\infty} \int_{\mathbb{R}^n} \frac{|u(x,t)|}{(1+|x|^{n+2s})(1+|t|^{1+\alpha})}\operatorname{d}\!x\operatorname{d}\!t<+\infty\right\}, $$
which is  an optimal   functions space  for defining a  solution of equation \eqref{B} in the sense of distributions.
%To this end , it is imperative  to derive a  key decay estimate of $\lt[D_{\rm{right}}^\alpha+(-\Delta)^s\rt]\varphi(x,t)$.
Then under    asymptotic assumption  \eqref{asym},
%\begin{equation*}\label{AA}
  %\liminf_{|x|\rightarrow\infty}\frac{u(x,t)}{|x|^\gamma}\geq 0 \; ( \mbox{or} \; \leq 0) \,\,\mbox{for some} \;0\leq\gamma\leq 1,
%\end{equation*}
% in the case $\frac{1}{2}<s < 1$,
 we prove that all distributional solutions $u\in\mathcal{L}_{2s,\alpha}(\mathbb{R}^n\times\mathbb{R})$ of equation \eqref{B} must be constant.

 %Some innovations  and breakthroughs of this paper are given below. First,
It is worth  emphasizing that the space $\mathcal{L}_{2s,\alpha}(\mathbb{R}^n\times\mathbb{R})$ is clearly a subspace of  $\mathcal{L}_{2s}(\mathbb{R}^n)\times\cL_\alpha^-(\mathbb{R})$ which is a natural space we usually used in the setting of classical entire solutions
%the following assumption
%  $$u(x,t)\in\left( C^{1,1}_{loc}(\R^n)\cap\cL_{2s}(\R^n)\right)\times\left(C_{\rm loc}^{0,1}(\R)\cap \cL^{-}_\alpha(\R)\right),$$
 % will be imposed in order
 to  guarantee that the singular integral in  \eqref{1.00} and \eqref{1.0}  are
well defined.     Here,  the slowly increasing function spaces $\cL_{2s}$ and $\cL^{-}_\alpha(\R)$ are defined respectively by

 $$\cL_{2s}(\R^n):=\lt\{v\in L^1_{loc}(\R^n) \ | \int_{\mathbb{R}^n} \frac{|v(x)|}{1+|x|^{n+2s}} dx < +\infty\rt\}$$
 and
 $$\cL^{-}_\alpha(\R):=\lt\{v\in L^1_{loc}(\R)\ | \int_{-\infty}^{t} \frac{|v(\tau)|}{1+|\tau|^{1+\alpha}} d\tau < +\infty \text{ for\ each}\  t\in\mathbb{R}\rt\}.$$

In the special case when    $u$ is only a function of $x$, i.e., $u(x,t)=u(x)$, then
\[\int_{-\infty}^{+\infty}\int_{\mathbb{R}^n} \frac{|u(x)|}{(1+|x|^{n+2s})(1+|t|^{1+\alpha})}\operatorname{d}\!x\operatorname{d}\!t\sim\int_{\mathbb{R}^n} \frac{|u(x)|}{1+|x|^{n+2s}}\operatorname{d}\!x.\]
  Thus the space $\mathcal{L}_{2s,\alpha}(\mathbb{R}^n\times\mathbb{R})$  is reduced to $\cL_{2s}(\R^n)$
and %$ \partial_t^\alpha u(x)=0,$  thus
equation \eqref{B} becomes  the  fractional elliptic equation \eqref{C}.
% \begin{equation}\label{C}
% (-\Delta)^s u(x)=0~\mb{in}~\R^n.\end{equation}
 Hence, our result  includes the Liouville theorem for $s$-harmonic functions established in \cite{CDL} as a special case.
Similarly, if $u(x,t)=u(t)$, then $\mathcal{L}_{2s,\alpha}(\mathbb{R}^n\times\mathbb{R})$ is consistent with the space
$$\mathcal{L}_{\alpha}(\R)=\left\{u(t) \in L^1_{\rm loc} (\mathbb{R}) \mid \int_{-\infty}^{+\infty} \frac{|u(t)|}{1+|t|^{1+\alpha}}\operatorname{d}\!t<\infty\right\}\subset\mathcal{L}_{\alpha}^{-}(\R),$$
and %$ \partial_t^\alpha u(x)=0,$  thus
  equation \eqref{B} turns into  one-sided  Marchaud fractional equation \eqref{D}.
 %As far as our knowledge, there is still a lack of research on the Liouville-type theorem  involving the one-sided Marchaud fractional derivative $\partial_t^\alpha$. In our recent works \cite{GMZ} and \cite{MGZ}, the authors employed two different methods to establish that the bounded solutions of equation \eqref{D} must be constant.
 In this case, our results imply that  the solution  $u\in\cL_\alpha(\R)$ of \eqref{C} in the sense of distributions are constant and  thus remove the bounded-ness condition
 in \cite{GMZ} and \cite{MGZ},  even the asymptotic condition \eqref{asym} is not needed. From this point of view, our finding is novel even restricted to one-sided  Marchaud fractional equation \eqref{D}.

Based on the Liouville type theorem for Marchaud fractional equation \eqref{D}, we further consider nonhomogeneous equation

\begin{equation}\label{nonhom}
 \partial_t^\alpha u(t)=f(t)\ ~\mb{in}\ ~\R.\end{equation}
 and establish the integral representation of the  solutions
\begin{equation}\label{repret}
u(t)=c_{\alpha}\int_{-\infty}^t\frac{f(\tau)}{(t-\tau)^{1-\alpha}}d\tau,\,\, \mb{with}\,\, c_\alpha:=\frac{1}{\Gamma(\alpha)}.
\end{equation}
\medskip

We are now ready to state the main result of this paper.
\begin{theorem}\label{Liouville}
Let $0<s,\alpha<1$ and $n\geq2$.
Suppose that
$u\in\mathcal{L}_{2s,\alpha}(\mathbb{R}^n\times\mathbb{R}) $
is a solution of
\begin{equation}\label{model-Lio}
\partial_t^\alpha u(x,t)+(-\Delta)^s u(x,t) =0, \; \,\,(x, t)\in \mathbb{R}^n\times \mathbb{R},
\end{equation}
 in the sense of distributions.

In the case $\frac{1}{2}<s < 1$, we assume additionally that
\begin{equation}\label{AA}
  \liminf_{|x|\rightarrow\infty}\frac{u(x,t)}{|x|^\gamma}\geq 0 \; ( \mbox{or} \; \leq 0) \,\,\mbox{for some} \;0\leq\gamma\leq 1.
\end{equation}
Then $u$ must be a constant.
\end{theorem}

We say that $u\in\mathcal{L}_{2s,\alpha}(\mathbb{R}^n\times\mathbb{R})$ is a solution of (\ref{model-Lio}) in the sense of distributions,  if
\begin{equation}\label{def-distribution soln}
\int_{-\infty}^{+\infty} \int_{\mathbb{R}^n} u(x,t)\overline{[D_{\rm{right}}^\alpha+(-\Delta)^s]\varphi(x,t)}\operatorname{d}\!x\operatorname{d}\!t = 0
\end{equation}
for any $\varphi\in \mathcal{S}(\mathbb{R}^n\times\mathbb{R})$,  the Schwartz space in $\R^n\tm\R$, and
$$
D_{\rm{right}}^\alpha\varphi(x,t)
:=C_{\alpha}\int_{t}^{+\infty}
  \frac{\varphi(x,t)-\varphi(y,\tau)}{(\tau-t)^{1+\alpha}}\operatorname{d}\!\tau
$$
denotes  the Marchaud fractional (right) derivative depending on the values of the function in the future time. In the subsequent section, we will explain that      $\mathcal{L}_{2s,\alpha}(\mathbb{R}^n\times\mathbb{R})$ is an appropriate and optimal distributional space to ensure that  \eqref{def-distribution soln} is well-defined.
\begin{remark}
It is noteworthy that  asymptotic assumption \eqref{AA} of $u$ in Liouville Theorem \ref{Liouville} is necessary. When this condition is violated, there are counter examples, such as the function $u(x,t)=x_1$, which is a nonconstant solution of dual fractional parabolic equation \eqref{model-Lio} and does not satisfy \eqref{AA}.
\end{remark}

\begin{remark}
In particular, if $u$ is bounded, then $u$ belongs to $\mathcal{L}_{2s,\alpha}(\mathbb{R}^n\times\mathbb{R})$ and also satisfies the asymptotic assumption \eqref{AA}. Therefore,   Theorem \ref{Liouville} further generalized the Liouville theorem  for the dual fractional parabolic equation in the whole space established  in  \cite{GMZ}.
\end{remark}
As discussed earlier, in the case where $u(x,t)=u(t)$,  Theorem \ref{Liouville} can be  reduced to the following Liouville type theorem involving a one-sided Marchaud fractional time
operator $\partial_t^\alpha$ without any additional assumptions.
\begin{corollary}\label{Liouville-M}
Let $0<\alpha<1$.
Suppose that
$u\in\mathcal{L}_{\alpha}(\R)$
is a solution of
\begin{equation}\label{model-Lio-M}
\partial_t^\alpha u(t) =0\; \,\,\mb{in}~  \mathbb{R},
\end{equation}
 in the following sense\begin{equation*}\label{def-distribution soln-M}
\int_{-\infty}^{+\infty}  u(t)\overline{D_{\rm{right}}^\alpha\varphi(t)}\operatorname{d}\!t = 0,~\forall \varphi\in \mathcal{S}(\R).
\end{equation*} Then $u$ must be a constant.
\end{corollary}

The proof of  Theorem \ref{Liouville} is based on Fourier analysis which involves the Fourier transform of
 distributions. It requires us  to find an optimal subspace of tempered distributions,
 $\mathcal{L}_{2s,\alpha}(\mathbb{R}^n\times\mathbb{R})$, in which fractional time
 derivatives $D_{\rm{left}}^\alpha$ and fractional Laplacian $(-\Delta)^s$ make sense.
 %and thus one can  define a  solution of equation \eqref{model-Lio} in the sense of distributions.
  To this end, it is imperative  to derive a key estimate on
the decay rate of $\lt[D_{\rm{right}}^\alpha+(-\Delta)^s\rt]\varphi(x,t)$ for the functions in Schwartz space.
\begin{lemma}\label{lm1}
If $\varphi(x,t)\in \mathcal{S}(\R^{n}\times \R),$ then $\lt[D_{\rm{right}}^\alpha+(-\Delta)^s\rt]\varphi\in C^{\infty}(\R^n\times \R)$ and
\begin{equation}\label{optimal}\Abs{\lt[D_{\rm{right}}^\alpha+(-\Delta)^s\rt]\varphi(x,t)}\leq \fr{C}{(1+\abs{x}^{n+2s})(1+\abs{t}^{1+\alpha})} ~\mb{in~} \R^n\tm\R.\end{equation}
\end{lemma}
  This sharp estimate  is a key ingredient and an important tool in the analysis of dual fractional parabolic equations. It follows more or less from observations in \cite{BMT} and \cite{Si}(who do not provide proofs). Detailed proofs were given in \cite{S} for the fractional Laplacian, and in \cite{SV} for the Marchaud fractional derivatives.

 To illustrate the optimality of the decay exponent in \eqref{optimal}, let us consider the following counterexample.

\begin{counterexample}
Let $\psi(x,t)=\eta(x)\phi(t)$ with $\eta\in  C_0^\infty(\R^n,[-1,0])$ and $\phi\in  C_0^\infty(\R,[0,1])$ are two cut-off functions satisfying
%we let the smooth cut-off functions $\phi(x)\in [0,1]$ and $\eta(t)\in[0,1]$ satisfying
\begin{subequations}
\begin{align*}
\eta(x)=
\begin{aligned}
\left\lbrace
\begin{aligned}
-1, & \ \ x\in \overline{B_{1}(0)} \,, \\
0, & \ \ x\not\in B_{2}(0)\,,
\end{aligned}
\right.
\end{aligned}
\quad\quad
\phi(t)=
\begin{aligned}
\left\lbrace
\begin{aligned}
1, & \ \ t\in[-1,1]\,, \\
0, & \ \ t\not\in (-2,2)\,.
\end{aligned}
\right.
\end{aligned}
\end{align*}
\end{subequations}
%\begin{equation*}
%\left\{\begin{array}{r@{\ \ }c@{\ \ }ll}
%-1\leq \eta\leq 0 &\mbox{in}&\ \ \R^n\,, \\[0.05cm]
% \eta= -1 &\mbox{in}&\ \ \overline{B_{1}(0)}\,, \\[0.05cm]
%\eta= 0 &\mbox{in}&\ \ B_{2}^c(0)\,. \\[0.05cm]
%\end{array}\right.
%\end{equation*}
%and
%\begin{equation*}
%\left\{\begin{array}{r@{\ \ }c@{\ \ }ll}
%0\leq \phi\leq 1 &\mbox{in}&\ \ \R\,, \\ [0.05cm]
% \phi= 1 &\mbox{in}&\ \ [-1,1]\,, \\ [0.05cm]
%\phi= 0 &\mbox{in}&\ \ (-\infty,-2]\cup[2,+\infty)\,. \\[0.05cm]
%\end{array}\right.
%\end{equation*}
 Clearly, $\psi\in\mathcal{S}(\R^n\tm\R)$. However,  for $(x,t)\in B_2^c(0)\tm[-1,1]$,
  \begin{eqnarray}\label{eg}\nonumber
\Abs{\lt[D_{\rm{right}}^\alpha+(-\Delta)^s\rt]\psi(x,t)}&=&\Abs{C_{\alpha}\eta(x)\jf_t^{+\infty}\fr{\phi(t)-\phi(\tau)}{(\tau-t)^{1+\alpha}}
\operatorname{d}\!\tau
+C_{n,s}\phi(t)\jf_{\R^n}\fr{\eta(x)-\eta(y)}{\abs{x-y}^{n+2s}}\operatorname{d}\!y}
\\ \nonumber
&=&C_{n,s}\jf_{\R^n}\fr{-\eta(y)}{\abs{x-y}^{n+2s}}\operatorname{d}\!y
\geq C_{n,s}\jf_{B_1(0)}\fr{1}{\abs{x-y}^{n+2s}}\operatorname{d}\!y\\
&\geq&\fr{C_0}{1+\abs{x}^{n+2s}},
\end{eqnarray}
where $C_0$ is a positive constant that depends only on $n$ and $s$.
This shows that  the decay exponent with respect to $x$ in estimate \eqref{optimal} cannot be improved. Indeed, if not, suppose that there exist some positive constants $C_1$ and $\delta_1$ such that
\[\Abs{\lt[D_{\rm{right}}^\alpha+(-\Delta)^s\rt]\psi(x,t)}\leq\fr{C_1}{\lt(1+\abs{x}^{n+2s+\delta_1}\rt)\lt(1+|t|^{1+\alpha}\rt)}.\]
Together with \eqref{eg},   we have
%\[\fr{C_0}{\lt(1+\abs{x}^{n+2s}\rt)}\leq\fr{C_1}{\lt(1+\abs{x}^{n+2s+\delta_1}\rt)\lt(1+|t|^{1+\alpha}\rt)}~\mb{in}~B_2^c(0)\tm[-1,1].\]
%It follows  that
\[\fr{C_1}{C_0}\geq\fr{\lt(1+\abs{x}^{n+2s+\delta_1}\rt)\lt(1+|t|^{1+\alpha}\rt)}{\lt(1+\abs{x}^{n+2s}\rt)}\geq C(1+\abs{t}^{1+\alpha})\lt(1+\abs{x}^{\delta_1}\rt),~(x,t)\in  B_2^c(0)\tm[-1,1].\]
which is a contradiction as $\abs{x}\to+\infty$. In analogy with above estimates, %we can also derive a contradiction if there exist some positive constants $C_2$ and $\delta_2$ such that
%\[\Abs{\lt[D_{\rm{right}}^\alpha+(-\Delta)^s\rt]\varphi(x,t)}\leq\fr{C_2}{\lt(1+\abs{x}^{n+2s}\rt)\lt(1+|t|^{1+\alpha+\delta_2}\rt)}~\mb{in}~B_1(0)\tm\{(-\infty,-2]\cup[2,+\infty)\}.\]
we can also prove that  the decay rate with respect to $t$ in estimate \eqref{optimal} can not be improved. Consequently, the decay exponent in \eqref{optimal} is optimal.
\end{counterexample}

As an application of the above Liouville theorem  for homogeneous  Marchaud fractional equation (Corollary \ref{Liouville-M}), we establish an  integral representation of  solutions to the corresponding nonhomegeneous ones.

\begin{theorem}\label{equivalence}
Let $u\in  C^{\alpha+\epsilon}_{ {\rm loc}}(\mathbb{R}) \cap \mathcal{L}_{\alpha}(\mathbb{R})$ be nonnegative and $f$ be nonnegative and locally bounded.
Then the pseudo-differential equation
\begin{equation*}
  \partial_t^{\alpha} u(t) =f(t),\,\, t\in  \mathbb{R},
\end{equation*}
is equivalent to the integral equation
\begin{equation*}
  u(t)=c_0+ c_{\alpha}\int_{-\infty}^{t}
  \frac{f(\tau)}{{(t-\tau)^{1-\alpha}}}\operatorname{d}\!\tau,
\end{equation*} for any nonnegative constant $c_0$.

Furthermore, if $f$ is a function depends also on $u$, then under the following  assumption:
 \begin{equation*}f(t,u)\geq C_2>0 \;\mbox{ for } \;u\geq C_1>0, \;\mbox{ uniformly in } t\in  \mathbb{R},\end{equation*}
 we must have $c_0=0$, i.e.,
 \begin{equation*}
  u(t)=c_{\alpha}\int_{-\infty}^{t}
  \frac{f(\tau,u(\tau))}{{(t-\tau)^{1-\alpha}}}\operatorname{d}\!\tau.
\end{equation*}\end{theorem}
The remaining of this paper is organized as follows.
In Section \ref{sec2},  our primary focus is on  defining  an action of the dual
fractional parabolic operator $D_{\rm{left}}^\alpha+(-\Delta)^s$  on
  distributions, as introduced in Definition \ref{dlfd}, by constructing an appropriate tempered distributional space.
  % $\cL_{2s,\alpha}(\R^n\tm\R)$.
 To achieve this, we establish  some fundamental properties related to the  operator
  $D_{\rm{left}}^\alpha+(-\Delta)^s$, including  the Fourier Transform identities (Lemma \ref{Fourier-derivative}) and  the   ``integration  by parts formula'' (Lemma \ref{lem2.2}).
  In Section \ref{sec3}, we prove our main results as stated in Theorem \ref{Liouville}
  and derive a key estimate on the decay rate of $\lt[D_{\rm{right}}^\alpha+(-\Delta)^s\rt]\varphi(x,t)$ (Lemma \ref{lm1}).
  In Section \ref{sec4}, as an application of the Liouville theorem, we establish an equivalence between the pseudo-differential equation involving Marchaud fractional derivatives and the corresponding integral equation.
 % which plays a crucial role in exploring the space $\cL_{2s,\alpha}(\R^n\tm\R)$.

 Throughout this paper, we use $C$, $c$, $C_i$, $c_i$ to denote  general constants whose values may vary from line to line.
\medskip

\section{Preliminaries}\label{sec2}

In this section, we first derive the Fourier transform identities and an  ``integration  by parts formula" for sufficiently smooth functions in the setting
 of the dual fractional operator $\partial_t^\alpha+(-\Delta)^s$. Then, by  combining an optimal
  estimates established in Lemma \ref{lm1}, we give the definition of the fractional derivative
  $\partial_t^\alpha+(-\Delta)^s$ of a tempered distribution by duality. One-sided distributional spaces for the left fractional Marchaud derivatives $\partial_t^\alpha$  and the definition of fractional powers of the right derivative was introduced first by  Stinga and Vaughan et. in \cite{SV,BMT}.

\subsection{Fourier Transform}\label{2.1}
We begin this section by recalling the following classical formula related with the Gamma function $\Gamma$.
\begin{equation}
\lm^\alpha=\fr{1}{\Gamma(-\alpha)}\int_0^{+\infty}(e^{-t\lm}-1)\fr{\operatorname{d}\!t}{t^{1+\alpha}},~0<\alpha<1, \lm\neq0.
\end{equation}
This absolutely convergent integral can be interpreted also as integral along the complex path
$\{z=t\lm:0<t<\infty\}$. By applying the Cauchy Integral Theorem, we are able to extend this form to complex parameters $i\lm$, that is,
\begin{equation}\label{Gamma}
(i\lm)^\alpha=\fr{1}{\Gamma(-\alpha)}\int_0^{+\infty}(e^{-i t\lm}-1)\fr{\operatorname{d}\!t}{t^{1+\alpha}},~0<\alpha<1,\lm\neq0.
\end{equation}
These formulas will be the keys %or key tools
  to derive
the  Fourier Transform of the  Marchaud fractional  derivative $D_{\rm{left}}^\alpha$ and $D_{\rm{right}}^\alpha$ acting on smooth functions as indicated by the following lemma, which has been separately proved  in \cite{BMT} for the Marchaud fractional derivatives $D_{\rm{left}}^{\alpha}$ and in \cite{Si} for the fractional Laplacian $(-\Delta)^s$.

\begin{lemma}\label{Fourier-derivative}
For any Schwartz function $\varphi\in \mathcal{S}(\R^n\tm\R)$, it holds that
\begin{equation*}
  \mathcal{F}\lt[\lt(D_{\rm{left}}^\alpha+(-\Delta)^s\rt)\varphi\rt](\xi,\rho)=\lt[(i\rho)^\alpha+\abs{\xi}^{2s}\rt]\mathcal{F}{\varphi}(\xi,\rho),
\end{equation*}
and
\begin{equation*}
  \mathcal{F}\lt[\lt(D_{\rm{right}}^\alpha+(-\Delta)^s\rt)\varphi\rt](\xi,\rho)=\lt[(-i\rho)^\alpha+\abs{\xi}^{2s}\rt]\mathcal{F}{\varphi}(\xi,\rho).
\end{equation*}
Where $\mathcal{F}\varphi(\xi,\rho)$  denotes the Fourier transform of the function $\varphi(x,t)$, that is,
\begin{equation}\label{f-transform}\mathcal{F}\varphi(\xi,\rho):=\fr{1}{(2\pi)^{\frac{n+1}{2}}}\jf_{-\infty}^{+\infty}\jf_{\R^n}e^{-i (\xi \cdot x+\rho t)}\varphi(x,t)\operatorname{d}\!x\operatorname{d}\!t.\end{equation}\end{lemma}

\subsection{Distributional setting}\label{2.2}
In this subsection,  we provide  the  integration  by parts formula for sufficiently smooth functions in the setting of the dual fractional operator $D_{\rm{left}}^\alpha+(-\Delta)^s$ based on the results in \cite{BMT} and \cite{Si}. With this special form as our starting point, we will subsequently extend  it to a class of  distributions  we are going to work with.
\begin{lemma}\label{lem2.2}
If $u,\varphi\in \mathcal{S}(\R^n\tm\R)$, then
\begin{equation*}\label{DIF}
  \int_{-\infty}^{+\infty}\int_{\R^n}\lt[D_{\rm{left}}^\alpha+(-\Delta)^s\rt]u(x,t)\overline{\varphi(x,t)}\operatorname{d}\!x\operatorname{d}\!t=
  \int_{-\infty}^{+\infty}\int_{\R^n}u(x,t)\overline{\lt[D_{\rm{right}}^\alpha+(-\Delta)^s\rt]\varphi(x,t)}\operatorname{d}\!x\operatorname{d}\!t.
\end{equation*}
\end{lemma}

In the continuation, we will utilize  Lemma \ref{lem2.2} to define $D_{\rm{left}}^\alpha+(-\Delta)^s$ in the sense of distributions. Notice that if $u\in \mathcal{S}',$ the dual space of $\mathcal{S}(\R^n\tm\R)$, then a natural definition would be
\begin{equation}\label{dual}
\lt\langle \lt[D_{\rm{left}}^\alpha+(-\Delta)^s\rt]u,\varphi\rt\rangle=\langle u,\lt[D_{\rm{right}}^\alpha+(-\Delta)^s\rt]\varphi\rangle,~\forall \varphi\in \mathcal{S}(\R^n\tm\R).\end{equation}
Nevertheless, it is straightforward from the Fourier transform identities in Lemma \ref{Fourier-derivative} to see that, in general, $\lt[D_{\rm{right}}^\alpha+(-\Delta)^s\rt]\varphi\notin \mathcal{S}(\R^n\tm\R)$ due to the fact % 一般due to 后面加the fact that
that the derivatives of multiplier $(-i\rho)^\alpha+\abs{\xi}^{2s}$ exhibit a singularity at the origin.
 In other words, the operator $D_{\rm{right}}^\alpha+(-\Delta)^s$ do not preserve the Schwartz class
 $\mathcal{S}(\R^n\tm\R)$. Hence we need to explore an alternative  distributional space which is
  a proper subset of $\mathcal{S}'$ to ensure the well-definedness of the right side of \eqref{dual}.
To this end, we define the class
 \[\mathcal{S}_{2s,\alpha}(\R^n\tm\R):=\left\{\varphi\in C^{\infty}(\R^n\tm\R)\mid\rho_{\gamma,\beta}(\varphi)<+\infty,\forall\mb{~multi-index~} \gamma \mb{~and~} \beta\in \mathbb{N}\right\},\]
 endowed with the  family of seminorms\begin{equation}\label{seminorm}\rho_{\gamma,\beta}(\varphi):=\sup\limits_{(x,t)\in \R^n\tm\R}(1+\abs{x}^{n+2s})(1+\abs{t}^{1+\alpha})|\partial_x^\gamma \partial_t^\beta \varphi(x,t)|,\end{equation}
 for each \mb{~multi-index~} $\gamma$ \mb{~and~} $\beta\in \mathbb{N}$.

Clearly, this  test functions space  includes the Schwartz space $\mathcal{S}(\R^n\tm\R)$
 and   one can easily derive that $\lt[D_{\rm{right}}^\alpha+(-\Delta)^s\rt]\varphi\in \mathcal{S}_{2s,\alpha}(\R^n\tm\R)$ for
 $\varphi\in \mathcal{S}(\R^{n}\times \R)$ from Lemma \ref{lm1}. %In other words, this  test function space  provides an optimal decay  rate  for $\lt[D_{\rm{right}}^\alpha+(-\Delta)^s\rt]\varphi.$
%\begin{lemma}\label{lm1}
%If $\varphi\in \mathcal{S}(\R^{n}\times \R),$ then $\lt[D_{\rm{right}}^\alpha+(-\Delta)^s\rt]\varphi\in \mathcal{S}_{2s,\alpha}(\R^n\tm\R)$.
%\end{lemma}
%This lemma is crucial for defining a  solution of equation \eqref{model-Lio} in the sense of distributions. It also provides an optimal  distribution  space for the  solution itself.  A comprehensive proof of this assertion  will be provided in the  Section 3.
Let ${\mathcal{S}}_{2s,\alpha}'$  be the dual space of ${\mathcal{S}}_{2s,\alpha}(\R^n\tm\R)$.  Then ${\mathcal{S}}_{2s,\alpha}'$ is a subspace of $\mathcal{S}'$ %owing to $\mathcal{S}(\R^n\tm\R)\subset\mathcal{S}_{2s,\alpha}(\R^n\tm\R)$.
and thus   the symmetry shown as in Lemma \ref{lem2.2} allows us to extend the definition of the operator $D_{\rm{left}}^\alpha+(-\Delta)^s$ to the
  space $\mathcal{S}_{2s,\alpha}'$  by duality.
  Namely,%Now we define the left fractional derivative of a distribution in  $ \mathcal{L}_{2s,\alpha}(\mathbb{R}^n\times\mathbb{R})$ as follows.
\begin{definition}\label{dlfd}
For  $u\in \mathcal{S}_{2s,\alpha}'$, we define $\lt[D_{\rm{left}}^\alpha+(-\Delta)^s\rt]u$ as a tempered distribution, given by
\begin{equation*}
  \lt\langle \lt[D_{\rm{left}}^\alpha+(-\Delta)^s\rt]u,\varphi\rt\rangle:=\langle u,\lt[D_{\rm{right}}^\alpha+(-\Delta)^s\rt]\varphi\rangle,~\forall \varphi\in \mathcal{S}(\R^n\tm\R).
\end{equation*}
\end{definition}
  %Namely, if $u\in\mathcal{S}_{2s,\alpha}'$, then $\lt[D_{\rm{left}}^\alpha+(-\Delta)^s\rt]u$ is the tempered distribution given by \eqref{dual}. By the way,  the notion $\langle\cdot,\cdot\rangle$ in \eqref{dual} and subsequent content represents a duality action, specifically,  we can regard any local integral function $u$ as a distribution by the following form\[\langle u,\varphi\rangle:=\int_{-\infty}^{+\infty} \int_{\mathbb{R}^n}u(x,t)\overline{\varphi(x,t)}\operatorname{d}\!x\operatorname{d}\!t.\]

 Moreover, $D_{\rm{left}}^\alpha+(-\Delta)^s$ is a continuous operator from $\mathcal{S}_{2s,\alpha}'$ to $\mathcal{S}'$. In particular,
  this operator will be defined in the following function space
  \begin{eqnarray*}
    \mathcal{L}_{2s,\alpha}(\mathbb{R}^n\times\mathbb{R}) &:=& L^1_{\rm loc} (\mathbb{R}^n\times\mathbb{R})\cap\mathcal{S}_{2s,\alpha}' \\
    &=& \left\{u:\mathbb{R}^n\times\mathbb{R}\to\R \mid \int_{-\infty}^{+\infty} \int_{\mathbb{R}^n} \frac{|u(x,t)|}{(1+|x|^{n+2s})(1+|t|^{1+\alpha})}\operatorname{d}\!x\operatorname{d}\!t<+\infty\right\}.
  \end{eqnarray*}
  This is an appropriate  distributional space in which one can  define  the distributional solutions  of equation \eqref{model-Lio}, as stated in  \eqref{def-distribution soln}.

By the way,   in Definition \ref{dlfd}, \eqref{dual} and subsequent content, the notion $\langle\cdot,\cdot\rangle$  represents a duality action. In general, for a locally integrable function $u$, such as $u\in\mathcal{L}_{2s,\alpha}(\mathbb{R}^n\times\mathbb{R})$,
we  regard it as a distribution in the following sense
\[\langle u,\varphi\rangle:=\int_{-\infty}^{+\infty} \int_{\mathbb{R}^n}u(x,t)\overline{\varphi(x,t)}\operatorname{d}\!x\operatorname{d}\!t.\]
for any $\varphi\in \mathcal{S}_{2s,\alpha}(\R^n\tm\R).$

%Now we define the left fractional derivative of a distribution as follows.
%\begin{definition}\label{dlfd}
%Let $u\in \mathcal{L}_{2s,\alpha}(\mathbb{R}^n\times\mathbb{R})$, we define $\lt[D_{\rm{left}}^\alpha+(-\Delta)^s\rt]u(x,t)$ by
%\begin{equation*}
 % \int_{-\infty}^{+\infty} \int_{\mathbb{R}^n} \lt[D_{\rm{left}}^\alpha+(-\Delta)^s\rt]u(x,t)\overline{\varphi(x,t)}\operatorname{d}\!x\operatorname{d}\!t=\int_{-\infty}^{+\infty} \int_{\mathbb{R}^n} u(x,t)\overline{\lt[D_{\rm{right}}^\alpha+(-\Delta)^s\rt]\varphi(x,t)}\operatorname{d}\!x\operatorname{d}\!t
%\end{equation*}
%for any $\varphi\in \mathcal{S}(\mathbb{R}^n\times\mathbb{R})$.
%\end{definition}

%At the end of this subsection, we racall the definition of the Fourier transform of tempered distributions in $\mathcal{L}_{2s,\alpha}(\mathbb{R}^n\times\mathbb{R})\subset\mathcal{S}',$ as provided in \cite{L}.
%\begin{definition}\label{dft}
%Let $u\in \mathcal{L}_{2s,\alpha}(\mathbb{R}^n\times\mathbb{R})$, we define $\mathcal{F}(u)$ by
%\begin{equation*}
%  \langle\mathcal{F}(u),\varphi\rangle:=\langle u,\mathcal{F}^{-1}(\varphi)\rangle\end{equation*}
%for any $\varphi\in \mathcal{S}(\mathbb{R}^n\times\mathbb{R})$, where \[\langle f,g\rangle:=\int_{-\infty}^{+\infty} \int_{\mathbb{R}^n}f(x,t)\overline{g(x,t)}\operatorname{d}\!x\operatorname{d}\!t.\]
%\end{definition}

\section{The proof of main results}\label{sec3}

In this section,
 we complete the proof of Liouville theorem for the homogeneous dual fractional parabolic equation in the
  whole space (Theorem \ref{Liouville}) and  derive a key estimate on the decay rate
  of $\lt[D_{\rm{right}}^\alpha+(-\Delta)^s\rt]\varphi$
 for functions in Schwartz space through
very delicate calculations (Lemma \ref{lm1}).

\subsection{The proof of Liouville theorem}\label{3.1}
We begin this subsection by recalling the definition of the Fourier transform of a tempered distribution $u$,
 particular in $\cL_{2s,\alpha}(\R^n\tm\R)\subset\mathcal{S}'$ (cf. \cite{L}),
\begin{equation}\label{tem}
  \langle\mathcal{F}u,\phi\rangle:=\langle u,\mathcal{F}^{-1}\phi\rangle,~\forall~\phi\in\mathcal{S}(\R^n\tm\R).
  \end{equation}
 Subsequently, we conclude that the solution  $u \in \cL_{2s,\alpha}(\R^n\tm\R)$ of equation
  \eqref{model-Lio} in the sense of distributions must be constant
by utilizing  a fundamental result from Harmonic Analysis (cf.\cite{grafok}), which states that
  \begin{equation}
  \label{fact}\mb{Let }~u\in\mathcal{S}'. \mb{~If~}~\mb{the~ support~of~} \mathcal{F}u ~\mb{is~a~ single~ point~(the origin), ~then}~ u ~\mb{must~ be~ a~ polynomial}.
  \end{equation}
\begin{proof}
[\bf Proof of Theorem \ref{Liouville}\,.] \,
Now for $u\in \cL_{2s,\alpha}(\R^n\tm\R)$, we first prove that
\begin{equation}\label{liou2}
  \langle\mathcal{F}u,\psi\rangle=0\,\,\mbox{for any}\,\, \psi\in C_0^\infty(\{\mathbb{R}^n\times\mathbb{R}\}\setminus \{({\bf 0}, 0)\}).
\end{equation}
 Indeed, for any $\psi\in C_0^\infty(\{\mathbb{R}^n\times\mathbb{R}\}\setminus \{({\bf 0}, 0)\})$,
 the function $\frac{\psi(\xi,\rho)}{(-i\rho)^\alpha+|\xi|^{2s }}$ also belongs
 to $C_0^\infty(\{\mathbb{R}^n\times\mathbb{R}\}\setminus \{({\bf 0}, 0)\})$.
  Since the Fourier transform $\mathcal{F}$ is a homeomorphism
  from $\mathcal{S}(\R^n\tm\R)$ onto $\mathcal{S}(\R^n\tm\R)$,
  there must exist a function $\varphi\in \mathcal{S}(\mathbb{R}^n\times\mathbb{R})$ such that
\begin{equation*}
  \mathcal{F}\varphi(\xi,\rho)=\frac{\psi(\xi,\rho)}{(-i\rho)^\alpha+|\xi|^{2s }}.
\end{equation*}
 It follows from Lemma \ref{lm1} that $\lt[D_{\rm{right}}^\alpha+(-\Delta)^s\rt]\varphi\in \mathcal{S}_{2s,\alpha}(\R^n\tm\R)\subset L^1(\R^n\tm\R)$.
 Then a combination of  \eqref{tem} and  Lemma \ref{Fourier-derivative} yields  that
\begin{eqnarray*}
 \langle\mathcal{F}u,\psi\rangle&=&\langle u,\mathcal{F}^{-1}\psi\rangle\\
 &=& \langle u,\mathcal{F}^{-1}\lt\{\lt[(-i\rho)^\alpha+|\xi|^{2s }\rt]\mathcal{F}\varphi(\xi,\rho)\rt\}\rangle \\
 &=& \langle u,\mathcal{F}^{-1}\lt\{\mathcal{F}\lt[\lt(D_{\rm{right}}^\alpha+(-\Delta)^s\rt)\varphi\rt](\xi,\rho)\rt\}\rangle \\
  &=& \langle u,\lt[D_{\rm{right}}^\alpha+(-\Delta)^s\rt]\varphi\rangle \\
  &=& \int_{-\infty}^{+\infty} \int_{\mathbb{R}^n}u(x,t)\overline{\lt[D_{\rm{right}}^\alpha+(-\Delta)^s\rt]\varphi(x,t)}\operatorname{d}\!x\operatorname{d}\!t=0.
\end{eqnarray*}
Here, we applied the assumption that  $u$ satisfies identity \eqref{def-distribution soln} in the last line.
 Hence, the assertion \eqref{liou2}  holds. That is to say,
$\mathcal{F}u$ is supported at the origin. Now, the fact \eqref{fact} infers that  $u(x,t)$ is a
 polynomial of $x$ and $t$ . While the condition $u\in\mathcal{L}_{2s,\alpha}(\mathbb{R}^n\times\mathbb{R})$ with $0<\alpha <1$
implies that this polynomial exhibits at most a linear growth and is independent of $t$, taking  the form:
$$u(x,t)=C_0+\sum_{i=1}^{n}C_ix_i,$$
where $C_i$ with $i=0,\,1,\,...,\,n$ are constants.
In the case where $0 < s \leq \frac{1}{2}$, we deduce from $u\in \cL_{2s}(\R^n)$ that $C_i = 0$ for $i=1,\,2,\,...,\,n$. While in the case $\frac{1}{2} < s < 1$, the asymptotic condition \eqref{AA} of $u$ also implies that $C_i = 0$ for $i=1,\,2,\,...,\,n$, and hence
$$u(x,t)\equiv C.$$
This completes the proof of Theorem \ref{Liouville}\,.
\end{proof}
\subsection{The proof of the optimal estimate}

In this subsection,  we show that if $\varphi\in \mathcal{S}(\R^{n}\times \R)$, then  $\lt[D_{\rm{right}}^\alpha+(-\Delta)^s\rt]\varphi\in \mathcal{S}_{2s,\alpha}(\R^n\tm\R)$.
 More precisely, we derive the following decay estimates
\begin{equation}\label{optimal1}\Abs{\lt[D_{\rm{right}}^\alpha+(-\Delta)^s\rt]\varphi(x,t)}\leq \fr{C}{(1+\abs{x}^{n+2s})(1+\abs{t}^{1+\alpha})} ~\mb{in~} \R^n\tm\R.\end{equation}

Although the similar estimates have been given separately for the fractional Laplacian $(-\Delta)^s$ in \cite{S} and for the Marchaud fractional derivatives $D_{\rm{right}}^{\alpha}$ in \cite{SV}, for readers' convenience, we follow their ideas to provide a proof for the dual fractional parabolic operator $(-\Delta)^s+D_{\rm{right}}^{\alpha}$ in the following.
 %Let us begin by  recalling the definition of  dual fractional parabolic operators
%\begin{eqnarray*}
% \lt[D_{\rm{right}}^\alpha+(-\Delta)^s\rt]\varphi(x,t)= C_{\alpha}\jf_t^{+\infty}\fr{\varphi(x,t)-\varphi(x,\tau)}{(\tau-t)^{1+\alpha}}\operatorname{d}\!\tau
% +C_{n,s} P.V.\jf_{\R^{n}}\fr{\varphi(x,t)-\varphi({y},t)}{\abs{x-y}^{n+2s}}\operatorname{d}\!{y}.
%\end{eqnarray*}
\begin{proof}[Proof of Lemma \ref{lm1}]
First we show that $\lt[D_{\rm{right}}^\alpha+(-\Delta)^s\rt]\varphi\in C^{\infty}(\R^n\times \R).$ A change of variables leads to
\begin{eqnarray*}
 \lt[D_{\rm{right}}^\alpha+(-\Delta)^s\rt]\varphi(x,t)= C_{\alpha}\jf_0^{+\infty}\fr{\varphi(x,t)-\varphi(x,\tau+t)}{\tau^{1+\alpha}}\operatorname{d}\!\tau
 +C_{n,s}P.V. \jf_{\R^{n}}\fr{\varphi(x,t)-\varphi(x+{y},t)}{\abs{y}^{n+2s}}\operatorname{d}\!{y}.
\end{eqnarray*}
Thus we easily obtain, for any multi-index $\gamma=(\gamma_1,...,\gamma_n) \mb{~and~} \beta\in\mathbb{N}$, %there holds that
\[\pl_x^{\gamma}\pl_t^\beta\lt\{\lt[D_{\rm{right}}^\alpha+(-\Delta)^s\rt]\varphi(x,t)\rt\}=\lt[D_{\rm{right}}^\alpha+(-\Delta)^s\rt]\lt(\pl_x^{\gamma}\pl_t^\beta\varphi(x,t)\rt).\]
It follows from $\varphi\in \mathcal{S}(\R^{n}\times \R)$ that
$\lt[D_{\rm{right}}^\alpha+(-\Delta)^s\rt]\varphi\in C^{\infty}(\R^n\times \R).$

In the sequel, % we only need to demonstrate that the seminorm $\rho_{\mathbf{0},0}\lt\{\lt[D_{\rm{right}}^\alpha+(-\Delta)^s\rt]\varphi\rt\} $ defined in \eqref{seminorm} is finite due to the fact that $\pl_x^{\gamma}\pl_t^\beta\varphi\in\mathcal{S}(\R^n\tm\R)$ for any multi-index$\gamma$ and $\beta\in\mathbb{N}$.
we  estimate the part of Marchaud derivative $D_{\rm{right}}^\alpha\varphi$, that is,
\begin{equation}\label{decay}\Abs{D_{\rm{right}}^\alpha\varphi(x,t)}\leq\fr{C}{(1+\abs{x}^{m})(1+\abs{t}^{1+\alpha})},~\forall (x,t)\in\R^n\tm\R,\forall m>0.\end{equation}
This provides an optimal decay rate for $D_{\rm{right}}^\alpha\varphi(x,t)$ with respect to $t$. Now we  proceed to estimate separately for the following two cases.
\\[0.3cm]\indent
$\mathbf{Case~ 1}$. $\abs{t}\leq 1.$ On account of $\varphi\in\mathcal{S}(\R^n\tm\R),$ applying the differential mean value theorem, there exists $h\in(0,1)$ such that
\begin{equation}\label{mean}
\Abs{\varphi(x,t)-\varphi(x,\tau+t)}\leq \Abs{\partial_t\varphi(x,t+h\tau)}|\tau|.
\end{equation}
Furthermore, we have the decay estimates
$$\abs{\varphi(x,t)},\abs{\partial_t\varphi(x,t)}\leq\fr{C}{1+\Abs{x}^{m}},\forall m>0.$$
Therefore, for any $k,m> 0,$ there holds that
\begin{eqnarray}\label{t<1}
 \Abs{D_{\rm{right}}^\alpha\varphi(x,t)}&\leq& \nonumber C_{\alpha}\jf_0^{+\infty}\fr{\Abs{\varphi(x,t)-\varphi(x,\tau+t)}}{\tau^{1+\alpha}}\operatorname{d}\!\tau
 \\ \nonumber
 &\leq&C_{\alpha}\lt\{\jf_0^{1}\fr{\Abs{\partial_t\varphi(x,t+h\tau)}}{\tau^{\alpha}}\operatorname{d}\!\tau+
 \jf_1^{+\infty}\fr{\Abs{\varphi(x,t)}+\Abs{\varphi(x,\tau+t)}}{\tau^{1+\alpha}}\operatorname{d}\!\tau\rt\}
\\
 %&\leq&\fr{C}{1+\abs{x}^{n+2s}}\\[0.5cm]
 &\leq&\fr{C}{(1+\abs{x}^{m})(1+|t|^{k})}.\end{eqnarray}
Here we have used the condition that
  $|t|\leq 1.$
\\[0.3cm]
\indent
$\mathbf{Case~ 2}$. $\abs{t}>1.$ Without loss of generality, we may assume $t<-1$ and write
\begin{eqnarray}\label{aa}
\nonumber&& \jf_0^{+\infty}\fr{\Abs{\varphi(x,t)-\varphi(x,\tau+t)}}{\tau^{1+\alpha}}\operatorname{d}\!\tau\\\nonumber
 &=&\jf_0^{-\frac{t}{2}}\fr{\Abs{\varphi(x,t)-\varphi(x,\tau+t)}}{\tau^{1+\alpha}}\operatorname{d}\!\tau+\jf_{-\frac{t}{2}}^{+\infty}\fr{\Abs{\varphi(x,t)-\varphi(x,\tau+t)}}{\tau^{1+\alpha}}\operatorname{d}\!\tau
=:I+II.\end{eqnarray}
Taking into account  of $\varphi, \partial_t\varphi\in\mathcal{S}(\R^n\tm\R),$ we derive
\begin{equation}\label{Schw}\abs{\varphi(x,t)},\abs{\partial_t\varphi(x,t)}\leq\fr{C}{\lt(1+\Abs{x}^{m}\rt)(1+|t|^{k})},\forall m,k> 0.\end{equation}
Which, together with   \eqref{mean} and  $\alpha<1, \abs{t}\leq 1$ , %together with 得放一起
deduce that
\begin{eqnarray}\label{bb}\nonumber
I&=& \jf_0^{-\frac{t}{2}}\fr{\Abs{\varphi(x,t)-\varphi(x,\tau+t)}}{\tau^{1+\alpha}}\operatorname{d}\!\tau\\  \nonumber
&\leq& \fr{C}{\lt(1+\Abs{x}^{m}\rt)}\jf_0^{-\frac{t}{2}}\fr{1}{(1+|t+h\tau|^{k+1-\alpha})\tau^\alpha}\operatorname{d}\!\tau.\\ \nonumber
 &\leq&
 %\fr{C}{\lt(1+\Abs{x}^{m}\rt)(1+|t|^{k+1-\alpha})}\jf_0^{-\frac{t}{2}}\fr{1}{\tau^\alpha}\operatorname{d}\!\tau.\\ \nonumber
 %&\leq&
  \fr{C}{\lt(1+\Abs{x}^{m}\rt)\lt(1+\abs{t}^{k}\rt)},
 %&\leq&
 % \fr{C}{\lt(1+\Abs{x}^{n+2s}\rt)(1+\abs{t}^{1+\alpha})}.
 \end{eqnarray}
holds for any $m,k> 0.$ Where we have utilized  the fact that $\abs{t+h\tau}\thicksim\abs{t} $ for $\tau\in(0,-\frac{t}{2}).$

On the other hand, it follows from \eqref{Schw} and the condition $\abs{t}\geq 1$ that, for  $\forall~k,m>0,$
%\[\varphi(x,t)\leq\fr{C}{\lt(1+\Abs{x}^{n+2s}\rt)(1+|t|^{1+\delta})}, ~\mb{for~ some~}\delta>0,\]
%which yields that
\begin{eqnarray*}\label{cc}\nonumber
II&\leq& |\varphi(x,t)|\jf_{-\frac{t}{2}}^{+\infty}\fr{1}{\tau^{1+\alpha}}\operatorname{d}\!\tau+\jf_{-\frac{t}{2}}^{+\infty}\fr{\Abs{\varphi(x,t+\tau)}}{\tau^{1+\alpha}}\operatorname{d}\!\tau\\  \nonumber
&\leq& \fr{C}{\lt(1+\Abs{x}^{m}\rt)\lt(1+|t|^{k}\rt)}\fr{1}{|t|^{\alpha}}+\fr{C}{\abs{t}^{1+\alpha}\lt(1+\Abs{x}^{m}\rt)}\jf_{-\infty}^{+\infty}\fr{1}{(1+|t+\tau|^{1+k})}\operatorname{d}\!\tau.\\
 &\leq&
 % \fr{C}{\lt(1+\Abs{x}^{n+2s}\rt)\abs{t}^{1+\alpha}} \\
 %&\leq&
  \fr{C}{\lt(1+\Abs{x}^{m}\rt)\lt(1+\abs{t}^{1+\alpha}\rt)}.
 \end{eqnarray*}
 Collecting these two terms, we get \begin{equation}\label{t>1}\Abs{D_{\rm{right}}^\alpha\varphi(x,t)}\leq\fr{C}{(1+\abs{x}^{m})(1+\abs{t}^{1+\alpha})},~\mb{for~} |t|\geq 1, x\in\R^n,\forall m>0.\end{equation}
Therefore, by considering   \eqref{t<1} and \eqref{t>1}, we have verified  estimate \eqref{decay}. %It's necessary to emphasize that  the optimal decay rate with respect to $t$,  $\frac{1}{1+\abs{t}^{1+\alpha}}$, is actually provided by the second integral in \eqref{cc}.

 By applying  a similar procedure  to the part of  fractional Laplacian $(-\Delta)^s\varphi$, one can derive an optimal decay rate with respect to $x$, $\frac{1}{1+\abs{x}^{n+2s}}$, more precisely,
 \begin{equation}\label{Schw1}\Abs{(-\Delta)^{s}\varphi(x,t)}\leq\fr{C}{(1+\abs{x}^{n+2s})(1+\abs{t}^{k})},~\forall (x,t)\in\R^n\tm\R, \forall k>0.\end{equation}
  Here, we omit the specific details of the proof.

In conclusion, together with   \eqref{bb} and \eqref{Schw1}, we obtain estimate \eqref{optimal1}  %yields that
%\begin{equation}\label{Schw2}\Abs{\lt[D_{\rm{right}}^\alpha+(-\Delta)^{s}\rt]\varphi(x,t)}\leq\fr{C}{(1+\abs{x}^{n+2s})(1+\abs{t}^{1+\alpha})},~\forall (x,t)\in\R^n\tm\R.\end{equation}
%that is, $\rho_{\mathbf{0},0}\lt\{\lt[D_{\rm{right}}^\alpha+(-\Delta)^s\rt]\varphi\rt\}<+\infty. $
and   thus complete the proof.
\end{proof}

\section{An application of the Liouville Theorem}\label{sec4}
In this final section,  we prove the equivalence of the pseudo-differential equation  \begin{equation}\label{pdeq}
  \partial_t^{\alpha} u =f\,\, \mb{in}\,\,  \mathbb{R},
\end{equation}
 and the integral equation
\begin{equation}\label{ieq}
  u(t)=c+ c_{\alpha}\int_{-\infty}^{t}
  \frac{f(\tau)}{{(t-\tau)^{1-\alpha}}}\operatorname{d}\!\tau.
\end{equation}
Based on the "Fractional Fundamental Theorem of Calculus" established in \cite{BMT} and \cite{SV}, if $u$ satisfies  the integral identity \eqref{ieq}, then it is a solution of  the pseudo-differential equation \eqref{pdeq}. Therefore,
it suffices to show that the solution of \eqref{pdeq} can be represented by the integral formula \eqref{ieq}.

Furthermore, if the nonlinearity $f$  depends also on $u$, under the following  assumption:
 \begin{equation}\label{asmf}f(t,u)\geq C_2>0 \;\mbox{ for } \;u\geq C_1>0, \;\mbox{ uniformly in } t\in  \mathbb{R},\end{equation}
 we will  show that the constant $c$  is exactly $0$ in \eqref{ieq}.
 %Let us begin by  recalling the definition of  dual fractional parabolic operators
%\begin{eqnarray*}
% \lt[D_{\rm{right}}^\alpha+(-\Delta)^s\rt]\varphi(x,t)= C_{\alpha}\jf_t^{+\infty}\fr{\varphi(x,t)-\varphi(x,\tau)}{(\tau-t)^{1+\alpha}}\operatorname{d}\!\tau
% +C_{n,s} P.V.\jf_{\R^{n}}\fr{\varphi(x,t)-\varphi({y},t)}{\abs{x-y}^{n+2s}}\operatorname{d}\!{y}.
%\end{eqnarray*}
\begin{proof}[Proof of Theorem \ref{equivalence}]
For each fixed $t\in\R.$ Choose $R>-t$ and define
\[v_R(s):=c_{\alpha}\int_{-\infty}^s\frac{f_R(\tau)}{(s-\tau)^{1-\alpha}}\operatorname{d}\!\tau, \,\, \mb{with}\,\, f_R(\tau)=\left\{\begin{array}{ll}f(\tau), &\tau\in {(-R,t)},\\
0,&\tau\in(-\infty,-R].\end{array}\right.\]
Clearly,
\begin{equation}\label{vr}\partial_t^{\alpha}v_R= f_R\,\, \mb{in}\,\, (-R,t].\end{equation}

Assume that $u$ is a nonnegative solution of pseudo-differential equation \eqref{pdeq}. Denote $$w_R:=u-v_R,$$ then by \eqref{pdeq} and \eqref{vr},  $w_R$ solves
\begin{equation}\label{weq}
\left\{
\begin{array}{ll}
    \partial_t^{\alpha} w_R=0 ,~   & \mb{in}\,\,  (-R,t], \\
  w_R=u\geq 0 , ~ &  \mb{in}\,\, (-\infty,-R].
\end{array}
\right.
\end{equation}
It follows from the Maximum Principle for the fractional Marchaud derivative $\partial_t^\alpha$ that
\begin{equation}\label{wr}
w_R\geq 0 \,\, \mb{in}\,\, (-R,t].
\end{equation}

Indeed, if \eqref{wr} is not valid, then there exists a point $t_0\in(-R,t]$ such that
\[w_R(t_0)=\inf\limits_{s\in(-\infty,t]}w_R(s)<0.\]
It follows by the definition of the one-sided nonlocal operator $\partial_t^\alpha$, we  derive
\[\partial_t^{\alpha}w_R(t_0)=C_{\alpha}\int_{-\infty}^{t_0}\fr{w_R(t_0)-w_R(\tau)}{(t_0-\tau)^{1+\alpha}}\operatorname{d}\!\tau<0.\]
This contradicts the first equation in \eqref{weq}. Thus, \eqref{wr} must be  valid.

In particular, we have
\[w_R(t)\geq 0,\, t\in\R,\]
that is,
\[u(t)\geq c_{\alpha}\int_{-\infty}^t\fr{f_R(\tau)}{(t-\tau)^{1-\alpha}}\operatorname{d}\!\tau.\]

Now Let $R\to\infty,$  then $f_R$ increases towards $f$ due to the fact that $f\geq 0.$ Thus, by the Levi's monotone convergence theorem, we obtain
\[u(t)\geq c_{\alpha}\int_{-\infty}^t\fr{f(\tau)}{(t-\tau)^{1-\alpha}}\operatorname{d}\!\tau.\]
This implies that the integral on the right hand side is well-defined, we call it $v(t)$.

Denote \[w(t)=u(t)-v(t),\]
then $w(t)\in \cL_{\alpha}(\R)$ satisfying
\begin{equation}
\left\{
\begin{array}{ll}
    \partial_t^{\alpha} w=f-  \partial_t^{\alpha} v=0 ,~   & \mb{in}\,\,  \R, \\
  w\geq  0 , ~ &  \mb{in}\,\, \R.
\end{array}
\right.
\end{equation}
Therefore, by the Liouville theorem for Marchaud fractional derivatives, Corollary \ref{Liouville-M}, we derive
\begin{equation*}
w(t)\equiv c,\, t\in\R,
\end{equation*}
 for any nonnegative constant $c$.

 Consequently,
 \begin{equation}\label{iequ}
u(t)=c+c_{\alpha}\int_{-\infty}^t\fr{f(\tau)}{(t-\tau)^{1-\alpha}}\operatorname{d}\!\tau,\, t\in\R.
\end{equation}

 Furthermore,  if the nonlinearity $f:=f(t,u)$ satisfies  \eqref{asmf}, then we claim that $c=0$.

 Indeed, if $c>0$, then $u\geq c$ due to the fact  that $f\geq 0.$ It follows from  \eqref{asmf} and \eqref{iequ}  that
 \begin{equation*}
u(t)\geq c_{\alpha}\int_{-\infty}^t\fr{f(\tau,u)}{(t-\tau)^{1-\alpha}}\operatorname{d}\!\tau\geq c_1\int_{-\infty}^t\fr{1}{(t-\tau)^{1-\alpha}}\operatorname{d}\!\tau=\infty,
\end{equation*}
which  contradicts the local  boundedness of  $u$. Thus, \eqref{iequ} becomes
 \begin{equation*}
u(t)=c_{\alpha}\int_{-\infty}^t\fr{f(\tau)}{(t-\tau)^{1-\alpha}}\operatorname{d}\!\tau,\, t\in\R.
\end{equation*}
Hence we complete the proof of Theorem \ref{equivalence}.
\end{proof}
{\bf{Funding:}}
The work of the first author is partially sponsored by the National Natural Science Foundation of China (Grant No.12501145, W2531006, 12250710674 and 12031012), the Natural Science Foundation of Shanghai (No.25ZR1402207),   the China Postdoctoral Science Foundation (No.2025T180838 and 2025M773061), the Postdoctoral Fellowship Program of CPSF (No. GZC20252004), and the Institute of Modern Analysis-A Frontier Research Center of Shanghai.

The work of
the second author is partially supported by the National Natural Science Foundation of China (NSFC Grant No.12101452).

The work of the third author is partially supported by the National Natural Science Foundation of China (NSFC Grant No.12071229).

{\bf Ethical approval:} Not applicable.

{\bf Informed consent:} Not applicable.

{\bf Author Contributions:} The authors have accepted responsibility for the entire content of this manuscript and approved its submission.

{\bf Data Availability Statement:} Not applicable.

{\bf Conflict of Interest:} Not applicable.

\bigskip

{\em Authors' Addresses  and E-mails:}
\medskip

{Yahong Guo}

School of Mathematical Sciences

Shanghai Jiaotong University

Shanghai, 200240, P.R. China

{yhguo@sjtu.edu.cn}
\medskip

Lingwei Ma

School of Mathematical Sciences

 Nankai University, Tianjin, 300071, P.R. China

{mlw1103@outlook.com}

\medskip

Zhenqiu Zhang

{School of Mathematical Sciences and LPMC

Nankai University

Tianjin, 300071, P.R. China

{zqzhang@nankai.edu.cn}

\end{document}